\documentclass{article}

\usepackage{dsfont}
\usepackage{polynom}
\usepackage{empheq}
\usepackage[latin1]{inputenc}
\usepackage{subfiles}
\usepackage{systeme}
\usepackage{CJK}
\usepackage{bm}
\usepackage{faktor}
\usepackage{ifthen}
\usepackage{mathrsfs}
\usepackage{amsmath}
\usepackage{mathdots}
\usepackage{mathtools}
\usepackage{amssymb}
\usepackage{amscd}
\usepackage{amstext}
\usepackage{appendix}
\usepackage{amsthm}
\usepackage{ytableau}
\usepackage{amsfonts}
\usepackage[all]{xy}
\usepackage{array}
\usepackage{enumerate}
\usepackage{graphicx}
\usepackage{hyperref}
\usepackage{longtable}
\usepackage[OT2,T1]{fontenc}
\usepackage{verbatim}
\usepackage{shuffle}
\usepackage{tikz}
\usetikzlibrary{matrix}
\usepackage{url,setspace,multirow,tikz-cd}
\usepackage[
backend=biber,
style=alphabetic,
sorting=ynt
]{biblatex}
\DeclareSymbolFont{cyrletters}{OT2}{wncyr}{m}{n}
\DeclareMathSymbol{\Sha}{\mathalpha}{cyrletters}{"58}

\theoremstyle{plain}
\newtheorem{thm}{Theorem}[section]
\newtheorem{lem}[thm]{Lemma}
\newtheorem{prop}[thm]{Proposition}
\newtheorem{cor}[thm]{Corollary}

\theoremstyle{definition}
\newtheorem{defn}[thm]{Definition}

\newtheorem{eg}[thm]{Example}

\theoremstyle{remark}
\newtheorem{rem}[thm]{Remark}

\newcommand{\ZZ}{{\mathbb Z}}
\newcommand{\QQ}{{\mathbb Q}}

\newcommand{\CC}{{\mathbb C}}

\newcommand{\NN}{{\mathbb N}}

\newcommand{\afk}{\mathfrak{a}}

\newcommand{\lfk}{\mathfrak{l}}
\newcommand{\mfk}{\mathfrak{m}}

\newcommand{\Lfk}{\mathfrak{L}}

\newcommand{\Hcal}{\mathcal{H}}

\newcommand{\Mcal}{\mathcal{M}}

\newcommand{\Ocal}{\mathcal{O}}
\newcommand{\Pcal}{\mathcal{P}}

\newcommand{\Tcal}{\mathcal{T}}

\newcommand{\Acal}{\mathcal{A}}

\newcommand{\isom}{\operatorname{Isom}}

\newcommand{\Tot}{\operatorname{Tot}}

\newcommand{\per}{\operatorname{per}}

\begin{document}

\title{Coproduct Formula for Motivic Version of Yamamoto's Integral}
\author{Ku-Yu Fan}
\date{}

\maketitle

\begin{abstract}
    Goncharov proved an explicit formula for the coproduct in the Hopf algebra of motivic iterated integrals. Yamamoto introduced Yamamoto's integral which generalizes iterated integrals and gave a new integral expression for multiple zeta star values using Yamamoto's integral. In this paper, we consider the motivic version of Yamamoto's integral and generalize Goncharov's coproduct formula to those motivic integrals. As an example, we will compute the coproduct of a certain type of Schur multiple zeta values.
\end{abstract}

\tableofcontents

\section{Introduction}
\subsection{Multiple zeta values}
Multiple zeta values are a special family of real numbers that appear in various mathematical fields such as the theory of mixed Tate motives, knot invariant, or evaluation of Feynman diagrams. In recent decades, numerous studies have been conducted on multiple zeta values. The history of multiple zeta values goes back to the time of Leonhard Euler who discovered the famous evaluation formula
$$\zeta(2) = \frac{1}{1^2} + \frac{1}{2^2} + \frac{1}{3^2} + \cdots = \frac{\pi^2}{6}.$$
More generally, he showed that the values of the Riemann zeta function at positive even integers are expressed as
$$\zeta(2k) = \sum_{0<n}\frac{1}{n^{2k}} = \frac{B_{2k}}{2(2k)!}(2\pi i)^{2k}$$
where $B_{2k}$ is the $2k$-th Bernoulli number.
Euler also tried to evaluate the values of the Riemann zeta function at a positive odd integer in terms of $\pi$, and although Euler failed to achieve the original objective, he found different types of evaluation formulas. Let
$$\zeta(k_1,k_2) = \sum_{0<n_1<n_2}\frac{1}{n_1^{k_1} n_2^{k_2}}$$
be a double sum analog of Riemann zate values.  Euler showed that
$$\zeta(3) = \zeta(1,2)$$
and analogous expressions for other Riemann zeta values.
Motivated by Euler's work Hoffman defined multiple zeta values (MZVs in short) as follows.
\begin{defn}[multiple zeta value]
    $$\zeta(k_1, \ldots, k_d)=\sum_{0<n_1<\cdots<n_d} \frac{1}{n_1^{k_1} \cdots n_d^{k_d}} \ \ \ \ (k_1,\ldots,k_{d-1}\in \ZZ_{>0}, k_d\in \ZZ_{>1})$$
\end{defn}
Hoffman started to investigate linear/algebraic relations with rational coefficients among MZVs and proved several families of MZV relations including a special case of the duality relation. A few years later, Kontsevich found an iterated integral expression for MZVs. For example,
$$\zeta(3) = \int_{0<t_1<t_2<t_3<1}\frac{dt_1}{1-t_1}\frac{dt_2}{t_2}\frac{dt_3}{t_3}.$$
In general, we define the iterated integral as follows.
\begin{defn}[iterated integral]
    For $a_0,\ldots, a_{k+1}\in \CC$ with $a_0\neq a_1, a_k\neq a_{k+1}$, and a piecewisely smooth path $\gamma:[0, 1]\rightarrow \CC$ from $a_0$ to $a_{k+1}$ such that $\gamma((0,1)) \subset \CC \setminus \{a_1, \ldots, a_k\}$ (that is, $\gamma$ does not pass through any points of $\{a_1, \ldots, a_k\}$), we define
    $$I_\gamma(a_0;a_1,\ldots,a_k; a_{k+1})$$
    by the iterated integral
    $$\int_{0<t_1<\cdots<t_k<1} \prod_{j=1}^{k}\frac{d\gamma(t_j)}{\gamma(t_j)-a_j} \in \CC.$$
\end{defn}
\begin{rem}
    When the $\gamma$ is not explicitly explained, $\gamma$ is a piecewisely smooth path from $a_0$ to $a_{k+1}$ which does not pass through any poles of differential forms.
\end{rem}
\begin{rem}
    When the $\gamma(t) = a_0(1-t) + a_{k+1}t$, we denote $I_\gamma(a_0;a_1,\ldots,a_k; a_{k+1})$ simply by $I(a_0;a_1,\ldots,a_k; a_{k+1})$.
\end{rem}
\begin{rem}
    MZVs are iterated integrals with $a_i\in \{0,1\}$, $a_0=0, a_1=1, a_k=0, a_{k+1}=1$ and $\gamma(t)=t$. More precisely,
    $$\zeta(k_1, \ldots, k_d) = (-1)^d I(0;1,\{0\}^{k_1-1},\ldots,1,\{0\}^{k_d-1}; 1).$$
\end{rem}

\subsection{Yamamoto's integral}
Yamamoto defined Yamamoto's integral associated with 2-posets, which generalizes the iterated integral and provides a simple representation for multiple zeta star values ([5]). Later, Hirose, Murahara, and Onozuka proved that Yamamoto's integral expression can also be used to express Schur multiple zeta values with constant entries on the diagonals ([4]). Here, we define Yamamoto's integral in a generalized setting and state its basic property ([5], Proposition 2.3).
\begin{defn}[labeled poset]
    $X = (X,\preceq,\delta)$ is called a labeled poset if $X$ is a finite set,
    $$\preceq \subset X\times X$$
    is a partial order, and
    $$\delta:X\to \CC$$
    is a map (this map is called a labeling map of $X$).\\
    We denote the corresponding strict partial order of $\preceq$ by $\prec$.
\end{defn}

A labeled poset $X$ is represented by a Hasse diagram with colored vertices. For example, the diagram
$$
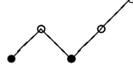
\begin{xy}
{(0,-4) \ar @{{*}-o} (4,0)},
{(4,0) \ar @{o-{*}} (8,-4)},
{(8,-4) \ar @{{*}-o} (12,0)},
{(12,0) \ar @{o-o} (16,4)}
\end{xy}$$
represents
\begin{align*}
    X & \coloneqq \{x_1, x_2, x_3, x_4, x_5\}\\
    \preceq & \coloneqq \{(x_1, x_2), (x_3, x_2), (x_3, x_4), (x_4, x_5), (x_3, x_5)\}\\
    \delta(x) & \coloneqq
    \begin{cases}
        0 & \text{ if } x = x_2, x_4, x_5\\
        1 & \text{ if } x = x_1, x_3.
    \end{cases}
\end{align*}

\begin{defn}\label{deflp}
    \begin{enumerate}[(1)]
        \item For $(X,\preceq,\delta)$ and $Y\subset X$, we define labeled subposet as
        $$Y = (Y,\preceq',\delta') \coloneqq (Y,\preceq \cap Y^2,\delta|_Y).$$
        \item For $X = (X,\preceq_X,\delta_X)$ and $Y = (Y,\preceq_Y,\delta_Y)$, we can define the direct sum
        $$X \sqcup Y = (X\sqcup Y, \preceq_{X\sqcup Y}, \delta_{X\sqcup Y}),$$
        where
        $$a\preceq_{X\sqcup Y} b \iff \begin{cases}
        a\preceq_{X} b \ \text{and} \ a,b\in X\\
        a\preceq_{Y} b \ \text{and} \ a,b\in Y,
        \end{cases}$$
        and
        $$\delta_{X\sqcup Y}(a) = \begin{cases}
        \delta_{X}(a) \ \text{if} \ a\in X\\
        \delta_{Y}(a) \ \text{if} \ a\in Y.
        \end{cases}$$
        \item A labeled poset $(X,\preceq,\delta)$ is said to be irreducible if
        $$X \neq Y \sqcup X_{\widehat{Y}},$$
        as a direct sum of posets for all non-empty proper subset $Y$ of $X$ ($\emptyset \subsetneq Y\subsetneq X$), where $X_{\widehat{Y}}$ is the subposet of $X$ restricted to $X\setminus Y$.
        \item For $(X,\preceq,\delta)$ and incomparable elements $a,b\in X$, we define
        $$X_a^b\coloneqq(X,\preceq \cup \{(x,y) |\ x\preceq a \text{ and } b\preceq y\},\delta).$$
        Notice that $X_a^b$ becomes a labeled poset.
        \item  For $(X,\preceq,\delta)$, we define
        $$\Tot(X) \coloneqq \{(X, \preceq', \delta)\ |\ \preceq \subset \preceq' \text{ and } \preceq' \text{ is a total order}\}.$$
    \end{enumerate}
\end{defn}

\begin{defn}[admissible labeled poset]
    For a path $\gamma:[0,1] \to \CC$ such that $\gamma((0,1))\subseteq \CC \setminus \delta(X)$, we say that $X$ is admissible (with respect to $\gamma$) if $\delta(x) \neq \gamma(1)$ for any maximal elements of $X$, and $\delta(x) \neq \gamma(0)$ for any minimal elements of $X$.
\end{defn}

\begin{defn}[Yamamoto's integral]\label{defYI}
    Let $\gamma$ be a path. For an admissible labeled poset $X$, Yamamoto's integral is defined by
    $$I_\gamma(X) \coloneqq \int_{\Delta(X)} \prod_{x\in X}\omega_{\delta(x)}(\gamma(t_x)),$$
    where
    $$\Delta(X) \coloneqq \{\mathbf{t} = (t_x)_{x\in X} \in (0,1)^X\ |\ t_x < t_y \text{ if } x\prec y\},$$
    $$\omega_{a}(t) \coloneqq \frac{dt}{t-a}.$$
\end{defn}

\begin{rem}
    The definition \ref{defYI} slightly differs from that in the original article ([5]), where the labeling map is
    $$\delta:X\to \{0,1\},$$
    the associated $1$-forms are
    $$\omega_0(t) = \frac{dt}{t} \text{ and } \omega_1(t) = \frac{dt}{1-t},$$
    and the path $\gamma:[0,1] \to \CC$ is defined by
    $$\gamma(t) = t.$$
\end{rem}

\begin{prop}\label{propYI}
    Let $X, Y$ be two labeled posets. For a path $\gamma:[0,1] \to \CC$ such that $\gamma((0,1))\subseteq \CC \setminus (\delta_X(X) \cup \delta_Y(Y))$ and incomparable elements $a,b\in X$, we have
    \begin{enumerate}[(1)]
        \item $I_\gamma(X)I_\gamma(Y) = I_\gamma(X\sqcup Y).$
        \item $I_\gamma(X) = I_\gamma(X_a^b)+I_\gamma(X_b^a).$
    \end{enumerate}
\end{prop}

\begin{proof}
    \begin{enumerate}[(1)]
        \item By Definition \ref{defYI}, we have
        \begin{align*}
            I_\gamma(X)I_\gamma(Y) & = \int_{\Delta(X)} \prod_{x\in X}\omega_{\delta_X(x)}(\gamma(t_x)) \int_{\Delta(Y)} \prod_{y\in Y}\omega_{\delta_Y(y)}(\gamma(t_y))\\
            & = \int_{\Delta(X) \times \Delta(Y)} \prod_{x\in X}\omega_{\delta_X(x)}(\gamma(t_x))\prod_{y\in Y}\omega_{\delta_Y(y)}(\gamma(t_y)).
        \end{align*}
        By Definition \ref{deflp}, we have
        $$\Delta(X) \times \Delta(Y) = \Delta(X\sqcup Y).$$
        Hence, we get
        $$I_\gamma(X)I_\gamma(Y) = \int_{\Delta(X\sqcup Y)} \prod_{x\in X\sqcup Y}\omega_{\delta_{X\sqcup Y}(x)}(\gamma(t_x)) = I_\gamma(X\sqcup Y).$$
        \item By Definition \ref{defYI}, we have
        \begin{align*}
            I_\gamma(X_a^b) + I_\gamma(X_b^a) & = \int_{\Delta(X_a^b)} \prod_{x\in X_a^b}\omega_{\delta_{X_a^b}(x)}(\gamma(t_x)) + \int_{\Delta(X_b^a)} \prod_{x'\in X_b^a}\omega_{\delta_{X_b^a}(x')}(\gamma(t_{x'}))\\
            & = \int_{\Delta(X_a^b)} \prod_{x\in X}\omega_{\delta_{X}(x)}(\gamma(t_x)) + \int_{\Delta(X_b^a)} \prod_{x'\in X}\omega_{\delta_{X}(x')}(\gamma(t_{x'}))\\
            & = \int_{\Delta(X_a^b)\sqcup \Delta(X_b^a)} \prod_{x\in X}\omega_{\delta_{X}(x)}(\gamma(t_x)).
        \end{align*}
        By Definition \ref{deflp}, we have
        $$\Delta(X_a^b) \sqcup \Delta(X_b^a) = \Delta(X).$$
        Hence, we get
        $$I_\gamma(X_a^b) + I_\gamma(X_b^a) = \int_{\Delta(X)} \prod_{x\in X}\omega_{\delta_X(x)}(\gamma(t_x)) = I_\gamma(X).$$
    \end{enumerate}
\end{proof}

Repeated use of (2) of Proposition \ref{propYI} yields an expression for Yamamoto's integral as a sum of iterated integrals:
$$I_\gamma(X) = \sum_{Y \in \Tot(X)} I_\gamma(Y).$$

\subsection{Motivic iterated integral}
Motive is a concept in algebraic geometry proposed by Alexander Grothendieck in the 1960s. In particular, the theory of mixed Tate motives has intricate and multifaceted relationships with MZVs. Deligne and Goncharov defined the motivic iterated integrals, which are elements in the commutative ring
$$\Pcal \coloneqq \Ocal(\isom_{\Mcal \Tcal(\QQ)}^\otimes (\omega_{\text{dR}}, \omega_{\text{B}}))$$
of periods obtained from the comparison isomorphism between the fiber functors $\omega_{\text{dR}}$ and $\omega_{\text{B}}$ of the Tannakian category $\Mcal \Tcal (\QQ)$ of mixed Tate motives over $\QQ$ (here, $\isom_{\Mcal \Tcal(\QQ)}^\otimes (\omega_{\text{dR}}, \omega_{\text{B}})$ denotes the affine scheme of tensor isomorphisms from $\omega_{\text{dR}}$ to $\omega_{\text{B}}$). It is known that there is a ring homomorphism
$$\per : \Pcal \rightarrow \CC$$
called period map, and period map $\per$ is speculated to be injective. For $k \geq 0$, $a_0, \ldots, a_{k+1} \in \QQ$ and a piecewisely smooth path $\gamma:[0, 1] \to \CC$ from $a_0$ to $a_{k+1}$ such that $\gamma((0,1)) \subset \CC \setminus \{a_1, \ldots, a_k\}$ and $\gamma'(0),\gamma'(1)\in \QQ\setminus \{0\}$, one can define the motivic iterated integral
$$I^\mfk_\gamma(a_0; a_1, \ldots, a_k; a_{k+1})$$
which is mapped to $I_\gamma(a_0; a_1, \ldots, a_k; a_{k+1})$ under the period map, and thus the motivic multiple zeta values are defined by
$$\zeta^\mfk(k_1, \ldots, k_d) \coloneqq (-1)^d I^\mfk_\gamma(0;1,\{0\}^{k_1-1},\ldots,1,\{0\}^{k_d-1}; 1).$$
Motivic iterated integral $I^m_\gamma$ satisfy various properties. Here, we list some of them which we will use in later arguments.
\begin{prop}\label{propMI}
    \begin{enumerate}[(1)]
        \item $I^\mfk_\gamma(a_0; a_1) = 1$.
        \item $I^\mfk_\gamma(a_0; a_1, \ldots, a_k; a_{k+1}) = (-1)^kI^\mfk_{\gamma^{-1}}(a_{k+1}; a_k, \ldots, a_1; a_0).$
        \item Motivic iterated integrals $I^\mfk_\gamma(a; \alpha; b)$ satisfy shuffle relation i.e.
        $$I^\mfk_\gamma(a; \alpha; b) \cdot I^\mfk_\gamma(a; \beta; b) = I^\mfk_\gamma(a; \alpha \shuffle \beta; b).$$
        \item Motivic iterated integrals $I^\mfk_{\gamma}(a; \omega; b)$ satisfy path composition formula i.e.
        $$I^\mfk_{\gamma \gamma'}(a; \omega; b) = \sum_{\alpha\beta = \omega} I^\mfk_{\gamma}(a; \alpha; c) \cdot I^\mfk_{\gamma'}(c; \beta; b),$$
        where $\gamma \gamma'$ is the composition of $\gamma$ and $\gamma'$.
    \end{enumerate}
\end{prop}
Now, we consider $(X,\preceq,\delta)$ with the labeling map $\delta: X\to \QQ$,  and define the motivic version of Yamamoto's integral as follows.
\begin{defn}\label{def}
    Let $X$ be a labeled poset with the labeling map
    $$\delta:X\to \QQ.$$
    Motivic version of Yamamoto's integral is defined by
    $$I^\mfk_\gamma(X) \coloneqq \sum_{Y \in \Tot(X)} I^\mfk_\gamma(Y),$$
    where $I^\mfk_\gamma(Y)$ is the motivic iterated integral, which corresponds to the iterated integral $I_\gamma(Y)$.
\end{defn}

\subsection{Coaction formula of motivic iterated integrals}
Let $\Acal$ be the quotient $\Pcal /\mu \Pcal$ and $\pi: \Pcal \rightarrow \Acal$ the natural projection to the quotient, where $\mu$ is the motivic version of $2\pi i$. By the theory of mixed Tate motives, there is a natural motivic coaction
$$\Delta : \Pcal \rightarrow \Acal \otimes \Pcal.$$
Let
$$\Hcal \coloneqq \langle I^\mfk_\gamma(a_0; a_1, \ldots, a_k; a_{k+1}) \rangle _{\QQ} \subset \Pcal$$
be the subspace spanned by motivic iterated integrals over $\QQ$. Then, $\Hcal$ becomes a $\QQ$-subalgebra by the shuffle product formula. Notice that the image $\pi(I^\mfk_\gamma(a_0; a_1, \ldots, a_k; a_{k+1}))$ depends only on the endpoints $a_0$ and $a_{k+1}$ of the path, so we drop $\gamma$ from the notation and denoted it as $I^\afk(a_0; a_1, \ldots, a_k; a_{k+1})$.
The coaction of $I^\mfk_\gamma(a_0; a_1, \ldots, a_k; a_{k+1})$ is given by the following formula.
\begin{thm}[|8|, Theorem 2.4; |2|, Theorem 1.2]
    We have
    \begin{align*}
        \Delta(I_\gamma^\mfk(a_0; a_1, \ldots, a_k; a_{k+1}))\\
        = \sum_{s=0}^{k} \sum_{\substack{i_0<\cdots<i_{s+1}\\ i_0 = 0, i_{s+1} = k+1}} \prod_{p=0}^{s} I^{\afk}(a&_{i_p};a_{i_p+1},\ldots,a_{i_{p+1}-1}; a_{i_{p+1}}))\otimes I_\gamma^{\mfk}(a_{i_0};a_{i_1},\ldots,a_{i_{s}}; a_{i_{s+1}}),
    \end{align*}
\end{thm}
It is known that $\Acal$ has a structure of graded algebra i.e., $\Acal = \bigoplus_{k \geq 0} \Acal_k$, and we let $\Acal_{>0} \coloneqq \bigoplus_{k > 0} \Acal_k$ and $\Lfk \coloneqq \Acal/\Acal_{>0}^2$. We denote by $\pi'$ the natural projection from $\Acal$ to $\Lfk$ and define
$$I^{\lfk}(a_{i_p};a_{i_p+1},\ldots,a_{i_{p+1}-1}; a_{i_{p+1}}) \coloneqq \pi' (I^{\afk}(a_{i_p};a_{i_p+1},\ldots,a_{i_{p+1}-1}; a_{i_{p+1}})).$$
\begin{defn}[|8|]
    For $r\in \NN$, the infinitesimal coaction $D_r: \Hcal \rightarrow \Lfk \otimes \Hcal$ is defined by
    \begin{align*}
        D_r(I_\gamma^\mfk(a_0; a_1, \ldots, a_k; a_{k+1}))\\
        = \sum_{s=0}^{k-r} I^{\lfk}(a_{s};a_{s+1},\ldots,a_{s+r};& a_{s+r+1}))\otimes I_\gamma^{\mfk}(a_{0};a_{1}, \ldots, a_{s}, a_{s+r+1}, \ldots ,a_{k}; a_{k+1}),
    \end{align*}
\end{defn}
In this article, we also consider the following.
\begin{defn}
    For $r\in \NN$, $D_r': \Hcal \rightarrow \Acal \otimes \Hcal$ is defined by
    \begin{align*}
        D_r'(I_\gamma^\mfk(a_0; a_1, \ldots, a_k; a_{k+1}))\\
        = \sum_{s=0}^{k-r} I^{\afk}(a_{s};a_{s+1},\ldots,a_{s+r};& a_{s+r+1}))\otimes I_\gamma^{\mfk}(a_{0};a_{1}, \ldots, a_{s}, a_{s+r+1}, \ldots ,a_{k}; a_{k+1}).
    \end{align*}
\end{defn}
\begin{rem}
    By viewing $\Delta$ and $D_r'$ as endomorphisms on $\Acal\otimes \Hcal$ that map $a\otimes h \in \Acal\otimes \Hcal$ to $(a\otimes 1)\cdot \Delta(h)$ and $(a\otimes 1)\cdot D_r'(h)$, they are related by
    $$\exp \left( \sum_{r = 1}^{\infty}D_r' \right) = \Delta.$$
\end{rem}

\section{Main theorems}
\subsection{Notations}
In order to state our main theorem, we will define some symbols. To describe the endpoints of the integrals, we define the extension of $X$ as follow.
\begin{defn}
    For a labeled poset $X = (X, \preceq, \delta)$ and a path $\gamma$, we define the extension of $X$ with respect to $\gamma$ as
    $$\widetilde{X}_{\gamma} \coloneqq (\widetilde{X}, \widetilde{\preceq}, \widetilde{\delta}),$$
    where
    \begin{align*}
        \widetilde{X} & \coloneqq X\cup \{x_0, x_1\}\\
        \widetilde{\preceq} & \coloneqq \preceq \cup \{(x_0, x)|x\in \widetilde{X}\} \cup \{(x, x_1)|x\in \widetilde{X}\}\\
        \widetilde{\delta}(x) & \coloneqq
        \begin{cases}
            \delta(x) & \text{ if } x\in X\\
            \gamma(0) & \text{ if } x = x_0\\
            \gamma(1) & \text{ if } x = x_1.
        \end{cases}
    \end{align*}
    When the path $\gamma$ is clear from the context, we drop the $\gamma$ from the notation.
\end{defn}

\begin{eg}
    Let $X = (X, \preceq, \delta)$ be the labeled poset
    \begin{align*}
        X & \coloneqq \{x_1, x_2, x_3\}\\
        \preceq & \coloneqq \{(x_1, x_2), (x_3, x_2)\}\\
        \delta(x) & \coloneqq
        \begin{cases}
            0 & \text{ if } x = x_2\\
            1 & \text{ if } x = x_1, x_3
        \end{cases}
    \end{align*}
    and $\gamma:[0,1] \rightarrow \CC$ be the path $\gamma(t) = t$.
    Then, $X$ and $\widetilde{X}$ are depicted as
    $$
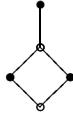
\begin{xy}
    {(0,-4) \ar @{{*}-o} (4,0)},
    {(4,0) \ar @{o-{*}} (8,-4)}
    \end{xy}$$
    and
    $$\begin{xy}
    {(0,-4) \ar @{{*}-o} (4,0)},
    {(4,0) \ar @{o-{*}} (8,-4)},
    {(0,-4) \ar @{{*}-o} (4,-8)},
    {(4,-8) \ar @{o-{*}} (8,-4)},
    {(4,0) \ar @{o-{*}} (4,5.656)}
    \end{xy}$$
    respectively.
\end{eg}
To describe the endpoints of the integrals, we define the following subsets.
\begin{defn}
    For posets $Y\subset X$, we define the subsets
    $$Y^{\prec X} \coloneqq Y^{\downarrow X} \setminus Y = \bigcup_{y\in Y} y^{\downarrow X} \setminus Y = \bigcup_{y\in Y} \{x\in X\,|\,x \preceq y\} \setminus Y$$
    $$Y^{\succ X} \coloneqq Y^{\uparrow X} \setminus Y = \bigcup_{y\in Y} y^{\uparrow X} \setminus Y = \bigcup_{y\in Y} \{x\in X\,|\,x \succeq y\} \setminus Y,$$
    and their subsets
    $$Y^{\rightarrow X} \coloneqq \{p\in Y^{\prec X}\ | \ \nexists\  l\in Y^{\prec X} \ \text{s.t.}\ p\prec l \}$$
    $$Y^{\leftarrow X} \coloneqq \{q\in Y^{\succ X}\ | \ \nexists\  u\in Y^{\succ X} \ \text{s.t.}\ u\prec q \},$$
    where $y^{\uparrow X}$ and $y^{\downarrow X}$ (resp. $Y^{\uparrow X}$ and $Y^{\downarrow X}$) denote the upper and lower closures of $y$ (resp. $Y$).\\
    In the following, we will use $y^{\succeq X}$ and $y^{\preceq X}$ (resp. $Y^{\succeq X}$ and $Y^{\preceq X}$)  instead of $y^{\uparrow X}$ and $y^{\downarrow X}$ (resp. $Y^{\uparrow X}$ and $Y^{\downarrow X}$).
    We will also use the following notations.
    \begin{enumerate}[(1)]
        \item We denote $\{y\}^{\prec X}$ (resp. $\{y\}^{\succ X}, \{y\}^{\rightarrow X}, \{y\}^{\leftarrow X}$) by $y^{\prec X}$ (resp. $y^{\succ X}, y^{\rightarrow X}, y^{\leftarrow X}$).
        \item We denote $x\in Y^{\rightarrow X}$ (resp. $x\in Y^{\leftarrow X}$) by $x \rightarrow_{X} Y$ or $Y \leftarrow_{X} x$ (resp. $x \leftarrow_{X} Y$ or $Y \rightarrow_{X} x$).
        \item When $Y^{\rightarrow X}$ (resp. $Y^{\leftarrow X}$) is a singleton we denote its unique element by $\rightarrow_{X} Y$ (resp. $\leftarrow_{X} Y$).
        \item When the set $X$ is clear, we denote $\rightarrow_{X}$ and $\leftarrow_{X}$ simply by $\rightarrow$ and $\leftarrow$.
    \end{enumerate}
\end{defn}

\begin{eg}
    Let $X = (X, \preceq, \delta)$ be the labeled poset
    \begin{align*}
        X & \coloneqq \{x_1, x_2, x_3, x_4, x_5\}\\
        \preceq & \coloneqq \{(x_1, x_2), (x_2, x_4), (x_4, x_5), (x_1, x_3), (x_3, x_4), (x_1, x_4), (x_2, x_5), (x_3, x_5), (x_1, x_5)\}\\
        \delta(x) & \coloneqq
        \begin{cases}
            0 & \text{ if } x = x_1, x_4\\
            1 & \text{ if } x = x_2, x_3, x_5,
        \end{cases}
    \end{align*}
    i.e.,
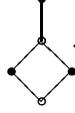
    $$\begin{xy}
    {(0,-4) \ar @{{*}-o} (4,0)},
    {(4,0) \ar @{o-{*}} (8,-4)},
    {(0,-4) \ar @{{*}-o} (4,-8)},
    {(4,-8) \ar @{o-{*}} (8,-4)},
    {(4,0) \ar @{o-{*}} (4,5.656)}
    \end{xy}.$$
    Let $Y$ be the subset $\{x_2, x_4\}$ of $X$. Then, we have $Y^{\prec X} = \{x_1, x_3\}$, $Y^{\succ X} = \{x_5\}$, $Y^{\rightarrow X} = \{x_3\}$ and $Y^{\leftarrow X} = \{x_5\}$.
\end{eg}
To state the formula for $D_r$, we define a set $X_r$ as follow.
\begin{defn}
    Let $X = (X, \preceq, \delta)$ be a labeled poset. For $r\in \NN$, we define
    $$X_r \coloneqq \{Y = (Y, \preceq_{Y}, \delta_{Y}) \subset X\ |\ |Y|=r, \ \forall\  x\in(X\setminus Y)\ \nexists\  y, y'\in Y \ \text{s.t.}\ y\preceq x\preceq y'\}.$$
\end{defn}
To state the right hand side of tensor product of formula, we define a set $X_{(p, \widehat{Y}, q)}$ as follow.
\begin{defn}
    Let $X = (X, \preceq, \delta)$ be a labeled poset. Fix $Y\in X_r$. For $p\rightarrow Y$ and $q\leftarrow Y$, we define
    $$X_{(p, \widehat{Y}, q)} \coloneqq (X_{\widehat{Y}}, \preceq_{(p, \widehat{Y}, q)}, \delta_{\widehat{Y}})$$
    where
    \begin{align*}
        X_{\widehat{Y}} \coloneqq & X\setminus Y\\
        \preceq_{(p, \widehat{Y}, q)} \coloneqq & \preceq \setminus \{(a,b)\in \preceq\,|\,a\in Y \text{ or } b\in Y\}\\
        & \cup \{(a,b)\in X_{\widehat{Y}}\times X_{\widehat{Y}}\,|\,a\in Y^{\prec X},\ p\preceq b\}\\
        & \cup \{(a,b)\in X_{\widehat{Y}}\times X_{\widehat{Y}}\,|\,b\in Y^{\succ X},\ a\preceq q\}\\
        \delta_{\widehat{Y}} \coloneqq & \delta|_{X_{\widehat{Y}}}
    \end{align*}
    When $|Y^{\rightarrow X}| = |Y^{\leftarrow X}| = 1$, we denote $X_{(p, \widehat{Y}, q)}$ simply by $X_{\widehat{Y}}$.\\
    We denote $X_{(p, \widehat{\{y\}}, q)}$ (resp. $X_{\widehat{\{y\}}}$) simply by $X_{(p, \widehat{y}, q)}$ (resp. $ X_{\widehat{y}}$).\\
\end{defn}

\begin{eg}
    Let $X$ be the labeled poset
    $$
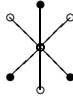
\begin{xy}
    {(0,-4) \ar @{{*}-o} (4,0)},
    {(4,0) \ar @{o-{*}} (8,-4)},
    {(4,0) \ar @{o-{*}} (4,5.656)},
    {(0,4) \ar @{o-o} (4,0)},
    {(4,0) \ar @{o-o} (8,4)},
    {(4,0) \ar @{o-o} (4,-5.656)}
    \end{xy}$$
    and $Y$ be the center of the diagram. For the bottom white vertex $p$ and the top black vertex $q$, $X_{(p, \widehat{Y}, q)}$ is depicted as follows:
    $$
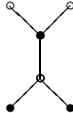
\begin{xy}
    {(0,-6.828) \ar @{{*}-o} (4,-2.828)},
    {(4,-2.828) \ar @{o-{*}} (8,-6.828)},
    {(4,-2.828) \ar @{o-{*}} (4,2.828)},
    {(0,6.828) \ar @{o-{*}} (4,2.828)},
    {(4,2.828) \ar @{{*}-o} (8,6.828)}
    \end{xy}$$
\end{eg}

\begin{defn}
    Let $X$ be a labeled poset. For $Y\in X_r$ and $\overline{Y} = (\overline{Y}, \preceq_{\overline{Y}}, \delta_{\overline{Y}}) \in \Tot(Y)$, we define
    $$\Tot_{\overline{Y}}(X) \coloneqq \{\overline{X} = (\overline{X}, \preceq_{\overline{X}}, \delta_{\overline{X}}) \in \Tot(X)\,|\,\overline{Y}\in \overline{X}_r, \preceq_{\overline{Y}} \subset \preceq_{\overline{X}}\}$$
    Notice that when $r=1$, we have $\Tot_{\overline{Y}}(X) = \Tot(X)$.\\
    Further for $p\rightarrow Y$ and $q \leftarrow Y$, we define
    $$\Tot_{p,\overline{Y},q}(X) \coloneqq \{\overline{X}\in \Tot_{\overline{Y}}(X) \,|\, \preceq_{(p,\widehat{Y},q)}\subset \preceq_{\overline{X}}\}$$
\end{defn}

\begin{defn}
Let $X$ be a labeled poset. For $Y \subset X$, we define $P_{Y}(X) = X_{\widehat{Y}}$.
\end{defn}

\begin{defn}
    For $Y\subset X$, $Y_1,Y_2\in \Tot(Y)$ and $\overline{X}\in \Tot_{Y_1}(X)$, we define
    $$T_{Y_1,Y_2}(\overline{X}) = (X,(\preceq_{\overline{X}}\setminus \preceq_{Y_1}) \cup \preceq_{Y_2},\delta).$$
    Notice that in general $T_{Y_1,Y_2}(\overline{X})$ is not a poset.
\end{defn}

\subsection{A formula for $D_r$}
\begin{lem}\label{lem1}
    Let $X$ be a labeled poset. For $Y\in X_r$ and $\overline{Y}\in \Tot(Y)$, $\{\Tot_{p,\overline{Y},q}(X)\}_{\substack{p\rightarrow Y \\ q \leftarrow Y}}$ forms a partition of $\Tot_{\overline{Y}}(X)$\\
\end{lem}

\begin{proof}
    We check the definition of a partition. Since $Y\in X_r$, $\Tot_{p,\overline{Y},q}(X)$ is nonempty. For all $\overline{X}\in \Tot_{\overline{Y}}(X)$, we have $\overline{X}\in \Tot_{p,\overline{Y},q}(X)$, where
    $$p \text{ is the unique element determined by } \overline{X} \text{ s.t. } p \rightarrow_{X} Y \text{ and } \forall\, l\in Y^{\prec X} \ l\preceq_{\overline{X}} p$$
    $$q \text{ is the unique element determined by } \overline{X} \text{ s.t. } q \leftarrow_{X} Y \text{ and } \forall\, u\in Y^{\succ X} \ q\preceq_{\overline{X}} u$$
    By above, since $p,q$ are unique, if $p\neq p'$ or $q\neq q'$, then $\Tot_{p,\overline{Y},q}(X) \cap \Tot_{p',\overline{Y},q'}(X) = \emptyset$.
\end{proof}

\begin{lem}\label{lem2}
    Let $X$ be a labeled poset. For $Y\in X_r$, $\overline{Y}\in \Tot(Y)$ and $\overline{X}\in \Tot_{p,\overline{Y},q}(X)$, we have
    \begin{enumerate}[(1)]
        \item $P_{\overline{Y}}$ is a function from $\Tot_{p,\overline{Y},q}(X)$ to $\Tot(X_{(p, \widehat{Y}, q)})$.
        \item $P_{\overline{Y}}$ is a surjection, and for $W = \overline{X_{(p, \widehat{Y}, q)}} \in \Tot(X_{(p, \widehat{Y}, q)})$ we have $$P_{\overline{Y}}^{-1}(W) = \{\overline{X}_{\overline{Y},x} \,|\, x \in p^{\succeq W}\cap q^{\prec W} \},$$
        where
        $$\overline{X}_{\overline{Y},x} \coloneqq (X, \preceq_{W} \cup \{(x',y) \,|\, x'\in x^{\preceq W},\ y\in Y\} \cup \{(y,x') \,|\, x'\in x^{\succ W},\ y\in Y\}\cup \preceq_{\overline{Y}}, \delta).$$
    \end{enumerate}
\end{lem}

\begin{proof}
    \begin{enumerate}[(1)]
        \item By definition of $\Tot_{p,\overline{Y},q}(X)$,
        $$\forall\, \overline{X}\in \Tot_{p,\overline{Y},q}(X),\ \preceq_{X_{(p, \widehat{Y}, q)}}\subset \preceq_{\overline{X}}\ \Rightarrow \ \preceq_{X_{(p, \widehat{Y}, q)}}\subset \preceq_{\overline{X}_{\widehat{Y}}}$$
        i.e. $\overline{X}_{\widehat{Y}}\in \Tot(X_{(p, \widehat{Y}, q)})$, so $P_{\overline{Y}}$ is a function from $\Tot_{p,\overline{Y},q}(X)$ to $\Tot(X_{(p, \widehat{Y}, q)})$.
        \item Let $W = \overline{X_{(p, \widehat{Y}, q)}} \in \Tot(X_{(p, \widehat{Y}, q)})$. Then there exists
        $$\overline{X}_{\overline{Y},p} \coloneqq ((X, \preceq_{W} \cup \{(x',y) \,|\, x'\in p^{\preceq W},\ y\in Y\} \cup \{(y,x') \,|\, x'\in p^{\succ W},\ y\in Y\}\cup \preceq_{\overline{Y}}), \delta_{X})$$
        s.t. $P_{\overline{Y}}(\overline{X}_{\overline{Y},p}) = W$. Moreover, we have
        $$P_{\overline{Y}}^{-1}(W) = \{\overline{X}_{\overline{Y},x} \,|\, x \in p^{\succeq W}\cap q^{\prec W} \}.$$
    \end{enumerate}
\end{proof}
We rewrite the definition of $D_r$ and $D_r'$ by our notation.
\begin{defn}
    Let $X$ be a labeled poset and $\widetilde{X}$ be its extension with respect to some path. For $p,q\in \widetilde{X}$ and $Y\subset X$, we define
    $$I^\afk_{(p;q)}(Y) \coloneqq \pi(I^\mfk_{\gamma}(Y)) \text{ with an arbitrary path } \gamma \text{ from } \widetilde{\delta}_{\widetilde{X}}(p) \text{ to } \widetilde{\delta}_{\widetilde{X}}(q)$$
    and
    $$I^\lfk_{(p;q)}(Y) \coloneqq \pi'(I^\afk_{(p;q)}(Y)).$$
\end{defn}
\begin{defn}\label{defD}
    Let $X$ be a totally ordered set. We define
    $$D_r' : \Hcal \to \Acal \otimes \Hcal$$
    and
    $$D_r = (\pi'\otimes id)\circ D_r': \Hcal \to \Lfk \otimes \Hcal.$$
    as
    $$D_r'(I^\mfk_\gamma(X)) \coloneqq \sum_{Y\in X_r} I^\afk_{(\rightarrow_{\widetilde{X}} Y;Y \rightarrow_{\widetilde{X}})}(Y)\otimes I^\mfk_\gamma(X_{\widehat{Y}})$$
    and
    $$D_r(I^\mfk_\gamma(X)) \coloneqq \sum_{Y\in X_r} I^{\lfk}_{(\rightarrow_{\widetilde{X}} Y;Y \rightarrow_{\widetilde{X}})}(Y)\otimes I^\mfk_\gamma(X_{\widehat{Y}})$$
    respectively.
\end{defn}

\begin{thm}\label{thm1}
    Let $X$ be a labeled poset and $\gamma$ be a path. Then,
    $$D_r(I^\mfk_\gamma(X)) = \sum_{\substack{Y\in X_r \\ Y \text{ is irr. }}} \sum_{\substack{p\rightarrow_{\widetilde{X}} Y \\ q \leftarrow_{\widetilde{X}} Y}} I^\lfk_{(p;q)}(Y)\otimes I^\mfk_\gamma(X_{(p, \widehat{Y}, q)}).$$
\end{thm}

\begin{proof}
    First, by Definition \ref{def} and \ref{defD}, we have
    $$D_r(I^\mfk_\gamma(X)) = \sum_{\overline{X}\in \Tot(X)} D_r(I^\mfk_\gamma(\overline{X})) = \sum_{\overline{X}\in \Tot(X)} \sum_{Z\in \overline{X}_r} I^\lfk_{(\rightarrow_{\widetilde{\overline{X}}} Z;Z \rightarrow_{\widetilde{\overline{X}}})}(Z)\otimes I^\mfk_\gamma(\overline{X}_{\widehat{Z}}).$$
    Since the summation is finite, we can sum over $\overline{X}$ first to get
    $$D_r(I^\mfk_\gamma(X)) = \sum_{Y\in X_r} \sum_{Z\in \Tot(Y)} \sum_{\overline{X}\in \Tot_{Z}(X)} I^\lfk_{(\rightarrow Z;Z \rightarrow)}(Z)\otimes I^\mfk_\gamma(\overline{X}_{\widehat{Z}}).$$
    By Lemma \ref{lem1}, we have
    $$D_r(I^\mfk_\gamma(X)) = \sum_{Y\in X_r} \sum_{Z\in \Tot(Y)} \sum_{\substack{p\rightarrow_{\widetilde{X}} Y \\ q\leftarrow_{\widetilde{X}} Y}} \sum_{\overline{X}\in \Tot_{p,Z,q}(X)} I^\lfk_{(\rightarrow Z;Z \rightarrow)}(Z) \otimes I^\mfk_\gamma(\overline{X}_{\widehat{Z}}).$$
    By Lemma \ref{lem2}, we have
    $$D_r(I^\mfk_\gamma(X)) = \sum_{Y\in X_r} \sum_{Z\in \Tot(Y)} \sum_{\substack{p\rightarrow Y \\ q\leftarrow Y}} \sum_{\overline{X}_{\widehat{Z}}\in \Tot(X_{(p, \widehat{Z}, q)})} \sum_{\overline{X}\in P_{Z}^{-1}(\overline{X}_{\widehat{Z}})} I^\lfk_{(\rightarrow Z;Z \rightarrow)}(Z) \otimes I^\mfk_\gamma(\overline{X}_{\widehat{Z}}).$$
    For fixed $p,q$, we use (5) of Proposition \ref{propMI} and the definition of $D_r$ to get
    \begin{align*}
        \sum_{\overline{X}\in P_{Z}^{-1}(\overline{X}_{\widehat{Z}})} I^\lfk_{(\rightarrow Z;Z \rightarrow)}(Z) \otimes I^\mfk_\gamma(\overline{X}_{\widehat{Z}}) & = \left( \sum_{\overline{X}\in P_{Z}^{-1}(\overline{X}_{\widehat{Z}})} I^\lfk_{(\rightarrow Z;Z \rightarrow)}(Z) \right) \otimes I^\mfk_\gamma(\overline{X}_{\widehat{Z}})\\
        & = \left( I^\lfk_{(p;p \rightarrow_{\widetilde{\overline{X}_{\widehat{Z}}}})}(Z) + \cdots + I^\lfk_{(\rightarrow_{\widetilde{\overline{X}_{\widehat{Z}}}} q; q)}(Z) \right) \otimes I^\mfk_\gamma(\overline{X}_{\widehat{Z}})\\
        & = I^\lfk_{(p;q)}(Z)\otimes I^\mfk_\gamma(\overline{X}_{\widehat{Z}}).
    \end{align*}
    Hence, we have
    \begin{align*}
        D_r(I^\mfk_\gamma(X)) & = \sum_{Y\in X_r} \sum_{Z\in \Tot(Y)} \sum_{\substack{p\rightarrow Y \\ q\leftarrow Y}} \sum_{\overline{X}_{\widehat{Z}}\in \Tot(X_{(p, \widehat{Z}, q)})} I^\lfk_{(p;q)}(Z) \otimes I^\mfk_\gamma(\overline{X}_{\widehat{Z}})\\
        & = \sum_{Y\in X_r} \sum_{Z\in \Tot(Y)} \sum_{\substack{p\rightarrow Y \\ q\leftarrow Y}} I^\lfk_{(p;q)}(Z) \otimes \left( \sum_{\overline{X}_{\widehat{Z}}\in \Tot(X_{(p, \widehat{Z}, q)})} I^\mfk_\gamma(\overline{X}_{\widehat{Z}}) \right)\\
        & = \sum_{Y\in X_r} \sum_{Z\in \Tot(Y)} \sum_{\substack{p\rightarrow Y \\ q\leftarrow Y}} I^\lfk_{(p;q)}(Z) \otimes I^\mfk_\gamma(X_{(p, \widehat{Z}, q)}).
    \end{align*}
    Since the underlying set of $Y$ and $Z$ are the same and $p,q$ depend only on $Y$, we have
    $$D_r(I^\mfk_\gamma(X)) = \sum_{Y\in X_r} \sum_{\substack{p\rightarrow Y \\ q\leftarrow Y}} \sum_{Z\in \Tot(Y)} I^\lfk_{(p;q)}(Z) \otimes I^\mfk_\gamma(X_{(p, \widehat{Y}, q)}).
    $$
    Notice that
    $$\sum_{Z\in \Tot(Y)} I^\lfk_{(p;q)}(Z) = \prod_{\substack{Y_i\subset Y \\ Y_i \text{ is irr.}}} I^\lfk_{(p;q)}(Y_i).$$
    If $Y$ is reducible, then
    $$\pi'\left(\prod_{\substack{Y_i\subset Y \\ Y_i \text{ is irr.}}} I^\lfk_{(p;q)}(Y_i)\right) = 0,$$
    hence
    $$D_r(I^\mfk_\gamma(X)) = \sum_{\substack{Y\in X_r \\ Y \text{ is irr. }}} \sum_{\substack{p\rightarrow Y \\ q \leftarrow Y}} I^\lfk_{(p;q)}(Y)\otimes I^\mfk_\gamma(X_{(p, \widehat{Y}, q)}).$$
\end{proof}

\begin{cor}
    Let $X$ be a labeled poset and $\gamma$ be a path. Then,
    $$D_1(I^\mfk_\gamma(X)) \coloneqq \sum_{x\in X} \sum_{\substack{p\rightarrow_{\widetilde{X}} x \\ q\leftarrow_{\widetilde{X}} x}} I^\afk_{(p;q)}(x)\otimes I^\mfk_\gamma(X_{\widehat{x}}).$$
\end{cor}

\begin{lem}\label{lem3}
    If $Y\in X_r$ for some $r$, then for $Y_1,Y_2\in \Tot(Y)$,
    $$T_{Y_1,Y_2} : \Tot_{Y_1}(X) \to \Tot_{Y_2}(X)$$
    is a bijection.
\end{lem}

\begin{proof}
     Let $\overline{X}\in \Tot_{Y_1}(X)$. Then, we have $\preceq_{Y_2} \subset ((\preceq_{\overline{X}}\setminus \preceq_{Y_1}) \cup \preceq_{Y_2})$. Since as a set $Y_1 = Y_2$ and
    $$\{(y, x)\in \preceq_{\overline{X}}\,|\,x\in (X\setminus Y_2),\ y\in Y_2\} \cup \{(x, y)\in \preceq_{\overline{X}}\,|\,x\in (X\setminus Y_2),\ y\in Y_2\} \subset ((\preceq_{\overline{X}}\setminus \preceq_{Y_1}) \cup \preceq_{Y_2}),$$
    $Y_2\in X_r$ under the order $((\preceq_{\overline{X}}\setminus \preceq_{Y_1}) \cup \preceq_{Y_2})$. Hence, $T_{Y_1,Y_2}(\overline{X}) \in \Tot_{Y_2}(X)$.\\
    Now, we swap $Y_1$ and $Y_2$. Then $T_{Y_2,Y_1} : \Tot(Y_2) \to \Tot(Y_1)$ is the inverse of $T_{Y_1,Y_2}$, i.e.,
    $$T_{Y_2,Y_1} \circ T_{Y_1,Y_2} = id_{\Tot_{Y_1}(X)} \text{ and } T_{Y_1,Y_2} \circ T_{Y_2,Y_1} = id_{\Tot_{Y_2}(X)}.$$
    Hence, $T_{Y_1,Y_2}$ is a bijection from $\Tot_{Y_1}(X)$ to $\Tot_{Y_2}(X)$.
\end{proof}

\begin{prop}\label{prop1-2}
    Let $X$ be a labeled poset and $\gamma$ be a path. Fix a totally ordered set $Z \in \Tot(Y)$. Then,
    $$D_r'(I^\mfk_\gamma(X)) = \sum_{Y\in X_r} \sum_{\overline{X}\in \Tot_{Z}(X)} \prod_{i=1}^{n_{Y}} I^\afk_{(\rightarrow_{\widetilde{\overline{X}}} Z;Z \rightarrow_{\widetilde{\overline{X}}})}(Y_i)\otimes I^\mfk_\gamma(\overline{X}_{\widehat{Z}}).$$
\end{prop}

\begin{proof}
    From the proof of Theorem \ref{thm1} and by Definition \ref{def} and \ref{defD}, we have
    $$D_r'(I^\mfk_\gamma(X)) = \sum_{Y\in X_r} \sum_{Z\in \Tot(Y)} \sum_{\overline{X}\in \Tot_{Z}(X)} I^\afk_{(\rightarrow_{\widetilde{\overline{X}}} Z;Z \rightarrow_{\widetilde{\overline{X}}})}(Z)\otimes I^\mfk_\gamma(\overline{X}_{\widehat{Z}}).$$
    By Lemma \ref{lem3}, we can fix a totally ordered set $Z\in \Tot(Y)$ first and change the order of the summation to get
    $$D_r'(I^\mfk_\gamma(X)) = \sum_{Y\in X_r} \sum_{\overline{X}\in \Tot_{Z}(X)}  \sum_{\overline{Y} \in \Tot(Y)} I^\afk_{(\rightarrow Z;Z \rightarrow)}(\overline{Y})\otimes I^\mfk_\gamma(\overline{X}_{\widehat{Z}}).$$
    Let $\{Y_i\}_i$ be the set of irreducible components of $Y$. Then,
    $$D_r'(I^\mfk_\gamma(X)) = \sum_{Y\in X_r} \sum_{\overline{X}\in \Tot_{Z}(X)} \prod_{\substack{Y_i\subset Y \\ Y_i \text{ is irr.}}} I^\afk_{(\rightarrow Z;Z \rightarrow)}(Y_i)\otimes I^\mfk_\gamma(\overline{X}_{\widehat{Z}}).$$
\end{proof}

\subsection{A formula for $\Delta$}
\begin{defn}
    Let $X$ be a labeled poset. For $x, x' \in X$, we say $x'$ is adjacent to $x$ if
    $$x \rightarrow x' \text{ or } x' \rightarrow x,$$
    and denote it as $x\leftrightarrow x'$.
    For $y \in Y \subset X$, we define the connected component of $y$ in $Y$ as
    $$C_{X,Y}(y) \coloneqq \{y'\in Y\,|\, \exists \, y_1, \ldots ,y_n \in Y \text{ s.t. } y = y_1 \leftrightarrow \cdots \leftrightarrow y_n = y'\}$$
    and the component set of $Y$ as
    $$C_X(Y) \coloneqq \{C_{X,Y}(y)\,|\,y\in Y\}.$$
    Notice that $C_X(Y)$ forms a partition of $Y$.
\end{defn}

\begin{lem}\label{lem4}
    Let $X$ be a labeled poset. For $Y\subset X$,
    \begin{enumerate}[(1)]
        \item $P_{Y}$ is a function from $\Tot(X)$ to $\Tot(X_{\widehat{Y}})$.
        \item $P_{Y}$ is a surjection.
    \end{enumerate}
\end{lem}

\begin{proof}
    \begin{enumerate}[(1)]
        \item By the definition of $\Tot(X)$, $\preceq_{\overline{X}_{\widehat{Y}}}$ is a total order and
        $$\preceq_{X_{\widehat{Y}}}\subset \preceq_{\overline{X}_{\widehat{Y}}} \text{ for } \overline{X}\in \Tot(X)$$
        we have $P_{Y}(\overline{X})\in \Tot(X_{\widehat{Y}})$ i.e. $P_Y$ is a function from $\Tot(X)$ to $\Tot(X_{\widehat{Y}})$.
        \item For $W = \overline{X_{\widehat{Y}}} \in \Tot(X_{\widehat{Y}})$, we will construct $\overline{X}\in \Tot(X)$ s.t. $P_Y(\overline{X}) = W$. Now for each $y\in Y$, we pick a minimal element of $(y^{\succ X}\setminus Y, \preceq_W)$ and denote it by $m_y$. Let $m = \{m_y\,|\,y\in Y\}$. For $x\in m$, we define $Y_x = \{y\in Y\,|\,m_y = x\}$. For each $x\in m$, we pick a $Z_x \in \Tot(Y_x)$ and define
        $$\preceq_{\overline{X}} = \preceq_W \cup \bigcup_{x\in m} \overline{\preceq_{Z_x}} \cup \bigcup_{\substack{x, x'\in m\\ x\prec_{W} x'}} \preceq_{Y_{x,x'}},$$
        where
        $$\overline{\preceq_{Z_x}} = \preceq_{Z_x} \cup \{(a, b)\,|\,a\in Z_x,\ b\in x^{\succeq W}\} \cup \{(a, b)\,|\,a\in x^{\prec W}, b\in Z_x\}$$
        $$\preceq_{Y_{x,x'}} = \{(a, b)\,|\,a\in Y_x,\ b\in Y_{x'}\}.$$
        Now for each $y\in Y$, we need to check $\{(p, y)\,|\,p\in y^{\preceq X}\} \subset \preceq_{\overline{X}}$ and $\{(y, q)\,|\,q\in y^{\succeq X}\} \subset \preceq_{\overline{X}}$.\\
        For $p\in y^{\preceq X} \setminus Z_{m_y}$, since $y^{\preceq X} \subset m_y\!^{\preceq X}$, we have
        $$\{(p, y)\,|\,p\in y^{\preceq X} \setminus Z_{m_y}\} \subset \preceq_{\overline{X}},$$
        and for $p\in y^{\preceq X} \cap Z_{m_y}$, since $Z_{m_y} \in \Tot(Y_{m_y})$, we have
        $$\{(p, y)\,|\,p\in y^{\preceq X} \cap Z_{m_y}\} \subset \preceq_{\overline{X}}.$$
        Similarly, since $Z_{m_y} \in \Tot(Y_{m_y})$, by definition of $m_y$, we have
        $$\{(y, q)\,|\,q\in y^{\succeq X}\} \subset \preceq_{\overline{X}}.$$
        Hence, $(X, \preceq_{\overline{X}},\delta) \in \Tot(X)$. By definition of $\preceq_{\overline{X}}$, we have $P_Y(\overline{X}) = W$ i.e., $P_{Y}$ is a surjection.
    \end{enumerate}
\end{proof}

\begin{defn}
    Let $X$ be a labeled poset. For any $Y\subset X$ and $\overline{X}\in \Tot(X)$, we define a relation on $P_{Y}^{-1}(\overline{X_{\widehat{Y}}})$ by
    $$\overline{X} \sim_T \overline{X}' \iff \overline{X} = T_{Y_1, Y_2}(\overline{X}')$$
    for some $Y_1, Y_2 \in \Tot(Y).$
\end{defn}

\begin{lem}\label{lem5}
    Let $X$ be a labeled poset. For any $Y\subset X$ and $\overline{X}\in \Tot(X)$, $\sim_T$ is an equivalence relation on $P_{Y}^{-1}(\overline{X_{\widehat{Y}}})$.
\end{lem}

\begin{proof}
    \begin{enumerate}[(1)]
        \item Reflexivity: For all $\overline{X}\in \Tot(X)$ we have $(Y, \preceq_{\overline{X}})\in \Tot(Y)$. This imply $\overline{X} = T_{Y,Y}(\overline{X})$.
        \item Symmetry: If $\overline{X} \sim_T \overline{X}'$, then $\overline{X} = T_{Y_1, Y_2}(\overline{X}')$. This implies $(Y,\preceq_{\overline{X}}) = Y_2$. Similarly, we take $(Y, \preceq_{\overline{X}'})\in \Tot(Y)$. Then, we have $T_{Y_2,Y}(\overline{X}) = \overline{X}'$.
        \item Transitivity: If $\overline{X} \sim_T \overline{X}'$ and $\overline{X}' \sim_T \overline{X}''$, then $\overline{X} = T_{Y_1, Y_2}(\overline{X}')$ and $\overline{X}' = T_{Y_1', Y_2'}(\overline{X}'')$. Similarly, we take $(Y, \preceq_{\overline{X}''})\in \Tot(Y)$. Then, we have $T_{Y,Y_2}(\overline{X}'') = \overline{X}$.
    \end{enumerate}
\end{proof}

\begin{defn}
    By Lemma \ref{lem5}, for $W \in \Tot(X_{\widehat{Y}})$ and $\overline{X}\in P_{Y}^{-1}(W)$, we define the equivalence class of $\overline{X}$ as
    $$[\overline{X}] \coloneqq \{\overline{X}' \in P_{Y}^{-1}(W) \,|\, \overline{X}' \sim_T \overline{X}\}$$
    and the corresponding partition of $P_{Y}^{-1}(W)$ as
    $$E_W = P_{Y}^{-1}(W)/\sim_T \coloneqq \{[\overline{X}]\,|\,\overline{X}\in P_{Y}^{-1}(W)\}.$$
\end{defn}
With our notations, Goncharov's formula can be rewritten as follows:
\begin{thm}[|8|, Theorem 2.4; |2|, Theorem 1.2]
    Let $X$ be a totally ordered set. Then, we have
    $$\Delta(I^\mfk_\gamma(X)) \coloneqq \sum_{Y\subset X} \prod_{Z\in C_{X}(Y)} I^\afk_{(\rightarrow_{\widetilde{X}} Z;Z \rightarrow_{\widetilde{X}})}(Z)\otimes I^\mfk_\gamma(X_{\widehat{Y}})$$
    with our notations.
\end{thm}

\begin{defn}
    Let $X$ be a labeled poset and $\gamma$ be a path. For $Y \subset X$, we define
    $$\Delta_Y(I^\mfk_\gamma(X)) \coloneqq \sum_{\overline{X}\in \Tot(X)} \prod_{Z\in C_{\overline{X}}(Y)} I^\afk_{(\rightarrow_{\widetilde{\overline{X}}} Z;Z \rightarrow_{\widetilde{\overline{X}}})}(Z)\otimes I^\mfk_\gamma(\overline{X}_{\widehat{Y}}).$$
\end{defn}

\begin{rem}\label{rem}
    Let $X$ be a labeled poset and $\gamma$ be a path. Then we have the following relation.
    $$\Delta(I^\mfk_\gamma(X)) = \sum_{Y\subset X} \Delta_Y(I^\mfk_\gamma(X)).$$
\end{rem}

\begin{thm}\label{thm2}
    Let $X$ be a labeled poset, $\gamma$ be a path and $Y \subset X$. If for all $Z\in C_X(Y)$, $Z\in X_{|Z|}$ and $|Z^{\rightarrow \widetilde{X}}| = |Z^{\leftarrow \widetilde{X}}| = 1$, then one has
    $$\Delta_Y(I^\mfk_\gamma(X)) = \prod_{Z\in C_X(Y)} I^\afk_{(p_Z;q_Z)}(Z)\otimes I^\mfk_\gamma(X_{\widehat{Y}}),$$
    where $p_Z \coloneqq \rightarrow_{\widetilde{X}} Z$ and $q_Z \coloneqq \leftarrow_{\widetilde{X}} Z$.
\end{thm}

\begin{proof}
    First, by definition of $\Delta_Y(I^\mfk_\gamma(X))$ and Lemma \ref{lem4} we have
    \begin{align*}
        \Delta_Y(I^\mfk_\gamma(X)) & = \sum_{\overline{X}\in \Tot(X)} \prod_{Z\in C_{\overline{X}}(Y)} I^\afk_{(\rightarrow_{\widetilde{\overline{X}}} Z;Z \rightarrow_{_{\widetilde{\overline{X}}}})}(Z)\otimes I^\mfk_\gamma(\overline{X}_{\widehat{Y}})\\
        & = \sum_{W \in \Tot(X_{\widehat{Y}})} \sum_{\overline{X}\in P_{Y}^{-1}(W)} \prod_{Z\in C_{\overline{X}}(Y)} I^\afk_{(\rightarrow Z;Z \rightarrow)}(Z)\otimes I^\mfk_\gamma(W).
    \end{align*}
    By Lemma \ref{lem5}, we have
    $$\Delta_Y(I^\mfk_\gamma(X)) = \sum_{W\in \Tot(X_{\widehat{Y}})} \sum_{e\in E_W} \sum_{\overline{X}\in e} \prod_{Z\in C_{\overline{X}}(Y)} I^\afk_{(\rightarrow Z;Z \rightarrow)}(Z)\otimes I^\mfk_\gamma(W).$$
    Note that $\overline{X} \in [\overline{X}']$ implies that $C_{\overline{X}}(Y) = C_{\overline{X}'}(Y)$, $\rightarrow_{\widetilde{\overline{X}}} Z = \rightarrow_{\widetilde{\overline{X}'}} Z$, and $\leftarrow_{\widetilde{\overline{X}}} Z = \leftarrow_{\widetilde{\overline{X}'}} Z$, so we may use the notation $C_e(Y) \coloneqq C_{\overline{X}}(Y)$ for $\overline{X} \in e$. Thus, we may swap the order of the summation and the product to find
    $$\Delta_Y(I^\mfk_\gamma(X)) = \sum_{W\in \Tot(X_{\widehat{Y}})} \sum_{e\in E_W} \prod_{Z\in C_e(Y)} \left( \sum_{Z' \in [Z]} I^\afk_{(\rightarrow_{\widetilde{e}} Z;Z \rightarrow_{\widetilde{e}})}(Z') \right) \otimes I^\mfk_\gamma(W),$$
    where $Z'\in [Z]$ means that $Z'$ is the set $Z$ with the new order determined by $\overline{X} \in e$ and the same labeling map i.e. $Z' = T_{Y_1,Y_2}(Z)$ for some $Y_1, Y_2 \in \Tot(Y)$. Since the order of $Z' \in [Z]$ is contains the order of $Y$, we have
    $$\Delta_Y(I^\mfk_\gamma(X)) = \sum_{W\in \Tot(X_{\widehat{Y}})} \sum_{e\in E_W} \prod_{Z\in C_e(Y)} \left( \sum_{Z' \in \Tot(Z,\preceq_Y)} I^\afk_{(\rightarrow Z;Z \rightarrow)}(Z') \right) \otimes I^\mfk_\gamma(W).$$
    Note that $\{V\cap Z \,|\, V\in C_X(Y)\} = C_Y(Z)$ and $Z = \bigsqcup_{S\in C_Y(Z)}S$. Thus, by proposition \ref{propYI}, we have
    \begin{align*}
        \Delta_Y(I^\mfk_\gamma(X)) & = \sum_{W\in \Tot(X_{\widehat{Y}})} \sum_{e\in E_W} \prod_{Z\in C_e(Y)} \left(  I^\afk_{(\rightarrow Z;Z \rightarrow)} \left( \bigsqcup_{V\in C_{X}(Y)}(V\cap Z) \right) \right) \otimes I^\mfk_\gamma(W)\\
        & = \sum_{W\in \Tot(X_{\widehat{Y}})} \sum_{e\in E_W} \prod_{Z\in C_e(Y)} \prod_{V\in C_{X}(Y)} I^\afk_{(\rightarrow Z;Z \rightarrow)}(V\cap Z)\otimes I^\mfk_\gamma(W).
    \end{align*}
    The products are finite products, so we can change the order of the products to get
    \begin{equation}
        \Delta_Y(I^\mfk_\gamma(X)) = \sum_{W\in \Tot(X_{\widehat{Y}})} \sum_{e\in E_W} \prod_{V\in C_{X}(Y)} \prod_{Z\in C_e(Y)} I^\afk_{(\rightarrow Z;Z \rightarrow)}(V\cap Z)\otimes I^\mfk_\gamma(W).\label{eq:00}
    \end{equation}
    For $W\in \Tot(X_{\widehat{Y}})$ and $V\in C_X(Y)$, since $|V^{\rightarrow \widetilde{X}}| = |V^{\leftarrow \widetilde{X}}| = 1$, we may let $\rightarrow_{\widetilde{X}} V = p_V$ and $\leftarrow_{\widetilde{X}} V = q_V$ and define
    $$B_{W,V} = \{x\in W \,|\, p_V\preceq_W x \prec_W q_V\}$$
    and
    $$V_W \coloneqq \left\{(V_x)_{x\in B_{W,V}} \,\left|\, \substack{\bigcup_{x\in B_{W,V}}V_x = V \text{ and } \\ \left(\bigcup_{x\in B_{W,V}} \preceq_{V_x}\right) \cup \left(\bigcup_{\substack{x,x'\in B_{W,V} \\ x\prec x'}} \{(a,b) \,|\, a\in V_x,\ b\in V_{x'}\}\right) = \preceq_V}
    \right.\right\}.$$
    Now, note that in the right hand side of (\ref{eq:00}), $C_{X}(Y)$ is independent of $e$. Also, since $V\in X_{|V|}$, $E_W$ has a 1-1 correspondence with $\prod_{V\in C_X(Y)} V_W$. Thus, we can change the order of summation $\sum_{e\in E_W}$ and the product $\prod_{V\in C_{X}(Y)}$ to get
    $$\Delta_Y(I^\mfk_\gamma(X)) = \sum_{W\in \Tot(X_{\widehat{Y}})} \prod_{V\in C_X(Y)} \sum_{(V_x)_x\in V_W} \prod_{x\in B_{W,V}} I^\afk_{(x;x \rightarrow_{\widetilde{W}})}(V_x)\otimes I^\mfk_\gamma(W).$$
    By (5) of Proposition \ref{propMI}, we have
    $$\Delta_Y(I^\mfk_\gamma(X)) = \sum_{W\in \Tot(X_{\widehat{Y}})} \prod_{V\in C_X(Y)} I^\afk_{(p_V;q_V)}(V)\otimes I^\mfk_\gamma(W)$$
    and by Definition \ref{def}, we get
    $$\Delta_Y(I^\mfk_\gamma(X)) = \prod_{V\in C_X(Y)} I^\afk_{(p_V;q_V)}(V)\otimes I^\mfk_\gamma(X_{\widehat{Y}}),$$
    which completes the proof.
\end{proof}

\section{Examples}
In this section, we will explicitly compute a few examples of Theorems \ref{thm1} and \ref{thm2}. Our first two examples are also special cases of (the motivic version of) Schur multiple zeta values (hereafter abbreviated as SMZV) due to a theorem by Hirose, Murahara and Onozuka. Thus, let us first give a definition of Schur multiple zeta values before going into the examples.
\begin{defn}[Young diagram]
    Let $\lambda$ be a finite subset of $\NN^2$.
    $\lambda$ is called a Young diagram if $\lambda$ satisfying
    $$(i+1,j) \in \lambda \text{ or } (i,j+1) \in \lambda \implies (i,j) \in \lambda \text{ for } (i,j) \in \NN^2.$$
\end{defn}
\begin{defn}[skew Young diagram]
    $\lambda = \lambda' \setminus \lambda''$ is called a skew Young diagram, if there exist Young diagrams $\lambda'$ and $\lambda''$ such that $\lambda' \subset \lambda''$.
\end{defn}
\begin{defn}[semi-standard Young tableaux]
    For a fixed skew Young diagram $\lambda$, we define the set of semi-standard Young tableaux $\text{SSYT}(\lambda)$ as the set of maps $f$ from $\lambda$ to $\NN$ satisfying the conditions:
    \begin{itemize}
        \item If $(i,j), (i,j+1) \in \lambda$ then $f(i,j) \leq f(i,j+1)$,
        \item If $(i,j), (i+1,j) \in \lambda$ then $f(i,j) < f(i+1,j)$.
    \end{itemize}
    We call a map $\mathbf{k}:\lambda \to \NN$ an index for $\lambda$. Furthermore, we say that $\mathbf{k}$ is admissible if $\mathbf{k}(i,j) \geq 2$ for any $(i,j) \in \lambda$ such that $(i+1,j) \notin \lambda$ and $(i,j+1) \notin \lambda$.
\end{defn}
\begin{defn}[Schur multiple zeta values]
    For an admissible index $\mathbf{k}$, the Schur multiple zeta value is defined by
    $$\zeta(\mathbf{k}) = \sum_{f\in \text{SSYT}(\lambda)} \frac{1}{\prod_{(i,j) \in \lambda}f(i,j)^{\mathbf{k}(i,j)}}.$$
\end{defn}
Hirose, Murahara, and Onozuka proved that Yamamoto's integral expression can also be used to represent SMZVs with constant entries on the diagonals.
\begin{thm}[|4|, Theorem 1.2]\label{thmSY}
    If $\mathbf{k}$ is an admissible index such that $\mathbf{k}(i,j) = \mathbf{k}(i+1,j+1)$ for all $(i,j),(i+1,j+1) \in \lambda$, then the SMZV $\zeta(\mathbf{k})$ has Yamamoto's integral expression, i.e. we have
    $$
    \zeta \left( \ {\footnotesize \ytableausetup{centertableaux, boxsize=1.2em}
    \begin{ytableau}
        \none & \none & \none & \none &  k_1   \\
        \none &  k_i & \none & \none[\iddots] & \none \\
        \none & \none & \none[\ddots] \\
        \none & \none[\iddots] & \none & k_i  \\
            k_r & \none & \none
    \end{ytableau}}
    \ \right)
    = I \left( \
    \begin{xy}
        {(0,-6) \ar @{{*}-o} (4,-2)},
        {(4,-2) \ar @{.o} (8,2)},
        {(0,-5) \ar @/^2mm/ @{-}^{k_1} (7,2)},
        {(8,2) \ar @{.} (12,6)},
        {(8,2) \ar @{.} (12,-2)}
    \end{xy}
    \cdots
    \begin{xy}
        {(32,4) \ar @{{*}-o} (36,8)},
        {(32,4) \ar @{{*}-o} (36,0)},
        {(32,-4) \ar @{{*}-} (36,0)},
        {(32,-4) \ar @{.} (36,-8)},
        {(32,-12) \ar @{{*}-o} (36,-8)},
        {(32,4) \ar @/_2mm/ @{-}_{l_i} (32,-11)},
        {(36,8) \ar @{.o} (40,12)},
        {(36,8) \ar @{.o} (40,4)},
        {(36,0) \ar @{.} (40,4)},
        {(36,0) \ar @{.} (40,-4)},
        {(36,-8) \ar @{.o} (40,-4)},
        {(32,5) \ar @/^2mm/ @{-}^{k_i} (39,12)},
    \end{xy}
    \cdots
    \begin{xy}
        {(44,2) \ar @{.} (48,-2)},
        {(44,-6) \ar @{.{*}} (48,-2)},
        {(48,-2) \ar @{-o} (52,2)},
        {(52,2) \ar @{.o} (56,6)},
        {(48,-1) \ar @/^2mm/ @{-}^{k_r} (55,6)},
    \end{xy}
    \ \right),$$
    where $l_i$ is the number of $k_i$'s.
\end{thm}
Let us define the motivic version of the SMZV with constant entries on the diagonal by Yamamoto's integral on the right-hand side of this theorem. In the following examples, we will make use of the facts
\begin{equation}
    I^\mfk(0;0;1) = I^\mfk(0;1;1) = I^\mfk(a;a_1, \ldots, a_k;a) = 0 \ \ \ (k\geq 1).\label{eq:01}
\end{equation}
\begin{eg}
    Let us consider the case
    $$(X,\preceq,\delta) =
    \begin{xy}
        {(0,-8) \ar @{{*}-o} (4,-4)},
        {(4,-4) \ar @{o.o} (8,0)},
        {(8,0) \ar @{o-{*}} (12,4)},
        {(0,-8) \ar @{{*}-{*}} (-4,-4)},
        {(-4,-4) \ar @{{*}-{*}} (0,0)},
        {(0,0) \ar @{{*}-o} (4,4)},
        {(4,4) \ar @{o.o} (8,8)},
        {(12,4) \ar @{{*}-o} (8,8)},
        {(8,0) \ar @{o.o} (4,4)},
        {(4,-4) \ar @{o-{*}} (0,0)},
        {(0,1) \ar @/^2mm/ @{-}^{n} (7,8)},
    \end{xy}.$$
    In other words,
    \begin{align*}
        X & \coloneqq \{x_1, \ldots, x_n, x_{1}', \ldots, x_{n}'\}\\
        \preceq & \coloneqq \{(x_a, x_b)\,|\,a\leq b\}\cup \{(x_{a}', x_{b}')\,|\,a\leq b\}\cup \{(x_a, x_{b}')\,|\,a\leq b\}\\
        \delta(x) & \coloneqq \begin{cases}
        1 \text{ if } x \in \{x_{1}', x_{2}', x_1, x_n\}\\
        0 \text{ others.}
    \end{cases}
    \end{align*}
    By Theorem \ref{thmSY}, we have
    $$\zeta^\mfk \left( \ {\footnotesize \ytableausetup{centertableaux, boxsize=1.2em}
    \begin{ytableau}
        n & 1\\
        1 & n
    \end{ytableau}}
    \ \right)
    = I^\mfk \left(
    \begin{xy}
        {(0,-8) \ar @{{*}-o} (4,-4)},
        {(4,-4) \ar @{o.o} (8,0)},
        {(8,0) \ar @{o-{*}} (12,4)},
        {(0,-8) \ar @{{*}-{*}} (-4,-4)},
        {(-4,-4) \ar @{{*}-{*}} (0,0)},
        {(0,0) \ar @{{*}-o} (4,4)},
        {(4,4) \ar @{o.o} (8,8)},
        {(12,4) \ar @{{*}-o} (8,8)},
        {(8,0) \ar @{o.o} (4,4)},
        {(4,-4) \ar @{o-{*}} (0,0)},
        {(0,1) \ar @/^2mm/ @{-}^{n} (7,8)},
    \end{xy}\right).$$
    Let us calculate $D_3'(I^\mfk(X))$ using Theorem \ref{thm1}. For convenience, we assume $n\geq 5$ and set $\delta(x_0) = 0, \delta(x_{n+1}') = 1$. First, notice that the irreducible elements in $X_3$ are either of the four types
    \begin{enumerate}[Type 1.]
        \item $Y_{1,m} = \{x_{m},x_{m+1},x_{m+2}\}$.
        \item $Y_{2,m} = \{x_{m}',x_{m},x_{m+1}\}$.
        \item $Y_{3,m} = \{x_{m}',x_{m+1}',x_{m+1}\}$.
        \item $Y_{4,m} = \{x_{m}',x_{m+1}',x_{m+2}'\}$.
    \end{enumerate}
    Thus, by Theorem \ref{thm1} and property of Motivic iterated integrals we have
    \begin{align*}
        D_3'(I(X)) = & I^\lfk_{(x_{0};x_{1}')}(Y_{1,1})\otimes I^\mfk(X_{(x_{0}, \widehat{Y_{1,1}}, x_{1}')}) +  I^\lfk_{(x_{0};x_{2}')}(Y_{2,1})\otimes I^\mfk(X_{(x_0, \widehat{Y_{2,1}}, x_{2}')})\\
        & + I^\lfk_{(x_{1}';x_{3}')}(Y_{2,2})\otimes I^\mfk(X_{(x_{1}', \widehat{Y_{2,2}}, x_{3}')}) + I^\lfk_{(x_{1}';x_4)}(Y_{2,2})\otimes I^\mfk(X_{(x_{1}', \widehat{Y_{2,2}}, x_4)})\\
        & + I^\lfk_{(x_1;x_3)}(Y_{3,1})\otimes I^\mfk(X_{(x_1, \widehat{Y_{3,1}}, x_3)}) + I^\lfk_{(x_{1}';x_4)}(Y_{3,2})\otimes I^\mfk(X_{(x_{1}', \widehat{Y_{3,2}}, x_4)})\\
        & + I^\lfk_{(x_{n-1};x_{n+1}')}(Y_{3,n-1})\otimes I^\mfk(X_{(x_{n-1}, \widehat{Y_{3,n-1}}, x_{n+1}')})\\
        & + I^\lfk_{(x_{n-2}';x_{n+1}')}(Y_{3,n-1})\otimes I^\mfk(X_{(x_{n-2}', \widehat{Y_{3,n-1}}, x_{n+1}')}).
    \end{align*}
    Using Theorem \ref{thm1} or Proposition \ref{prop1-2} is a easy way to find formula of $D_r$ or $D_r'$ for some type of SMZVs.
\end{eg}

\begin{eg}
    Let us consider the case
    $$(X,\preceq,\delta) =
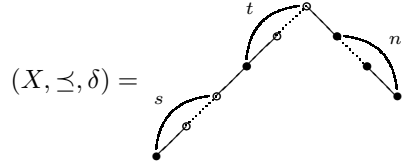
    \begin{xy}
        {(0,-10) \ar @{{*}-o} (4,-6)},
        {(4,-6) \ar @{o.o} (8,-2)},
        {(8,-2) \ar @{o-{*}} (12,2)},
        {(12,2) \ar @{{*}-o} (16,6)},
        {(16,6) \ar @{o.o} (20,10)},
        {(20,10) \ar @{o-{*}} (24,6)},
        {(24,6) \ar @{{*}.{*}} (28,2)},
        {(28,2) \ar @{{*}-{*}} (32,-2)},
        {(0,-9) \ar @/^2mm/ @{-}^{s} (7,-2)},
        {(12,3) \ar @/^2mm/ @{-}^{t} (19,10)},
        {(25,6) \ar @/^2mm/ @{-}^{n} (32,-1)}
    \end{xy}.$$
    In other words,
    \begin{align*}
        X & \coloneqq \{x_1, \ldots, x_s, x_{s+1}, \ldots, x_{s+t}, x_{1}', \ldots, x_{n}'\}\\
        \preceq & \coloneqq \{(x_a, x_b)\,|\,a\leq b\}\cup \{(x_{a}', x_{b}')\,|\,a\leq b\}\cup \{(x_{a}', x_{s+t})\,|\,a \in \{1, \ldots, n\}\}\\
        \delta(x) & \coloneqq \begin{cases}
        1 \text{ if } x \in \{x_{1}', \ldots, x_{n}', x_1, x_{s+1}\}\\
        0 \text{ others.}
    \end{cases}
    \end{align*}
    By Theorem \ref{thmSY}, we have
    $$I^\mfk \left(
    \begin{xy}
        {(0,-10) \ar @{{*}-o} (4,-6)},
        {(4,-6) \ar @{o.o} (8,-2)},
        {(8,-2) \ar @{o-{*}} (12,2)},
        {(12,2) \ar @{{*}-o} (16,6)},
        {(16,6) \ar @{o.o} (20,10)},
        {(20,10) \ar @{o-{*}} (24,6)},
        {(24,6) \ar @{{*}.{*}} (28,2)},
        {(28,2) \ar @{{*}-{*}} (32,-2)},
        {(0,-9) \ar @/^2mm/ @{-}^{s} (7,-2)},
        {(12,3) \ar @/^2mm/ @{-}^{t} (19,10)},
        {(25,6) \ar @/^2mm/ @{-}^{n} (32,-1)}
    \end{xy}\right)
    = \zeta^\mfk \left( \ {\footnotesize \ytableausetup{centertableaux, boxsize=1.2em}
    \begin{ytableau}
        \none & \none & \none & s\\
        1 & \cdots & 1 & t
    \end{ytableau}}
    \ \right).$$
    Let us calculate $\Delta(I^\mfk(X))$ using Theorem \ref{thm2}. For convenience, we set $\delta(x_0) = 0, \delta(x_{s+t+1}) = 1$. First, using Remark \ref{rem} we get
    $$\Delta(I^\mfk(X)) = \sum_{Y\subset X} \Delta_Y(I^\mfk(X)).$$
    Next, using Theorem \ref{thm2}, Proposition \ref{propYI}, and (\ref{eq:01}), we see that $\Delta_Y$ vanishes except for the following three types of $Y$'s.
    \begin{enumerate}[Type 1.]
        \item $Y_{1,m} = \{x_m, \ldots, x_{s+t}, x_{1}', \ldots, x_{n}'\}$, where $1\leq m \leq s+t$.
        \item $Y_{2,m} = \{x_1, \ldots, x_s, x_{s+m}, \ldots, x_{s+t}, x_{1}', \ldots, x_{n}'\}$, where $2\leq m \leq t$.
        \item $Y_{3,m_1,m_2} = \{x_2, \ldots, x_{s+m_1}, x_{s+m_2}, \ldots, x_{s+t}, x_{1}', \ldots, x_{n}'\}$, \\
        where $1\leq m_1 \leq t-2$ and $m_1+1 < m_2 \leq t$.
    \end{enumerate}
    Hence, we can write the formula of $\Delta(I^\mfk(X))$ by three types of $Y$
    $$\Delta(I^\mfk(X)) = \sum_{m = 1}^{s+t} I^\afk_{(x_{m-1};x_{s+t+1})}(Y_{1,m})\otimes I^\mfk(X_{\widehat{Y_{1,m}}})$$
    $$+ \sum_{m = 2}^{t} I^\afk_{(x_0;x_s)}(X_{1,s})I^\afk_{(x_{s+m-1};x_{s+t+1})}(Y_{2,m}\setminus X_{1,s})\otimes I^\mfk(X_{\widehat{Y_{2,m}}})$$
    $$+ \sum_{m_1 = 1}^{t - 2}\sum_{m_2 = m_1 + 2}^{t} I^\afk_{(x_1;x_{s+m_1+1})}(X_{2,s+m_1})I^\afk_{(x_{s+m_2-1};x_{s+t+1})}(Y_{3,m_1,m_2}\setminus X_{2,s+m_1})\otimes I^\mfk(X_{\widehat{Y_{3,m}}}).$$
    where $X_{a,b} = \{x_a, x_{a+1}, \ldots, x_{b-1}, x_b\}$.
    By Theorem \ref{thmSY}, we can compute the motivic coaction of some types of SMZVs using Theorem \ref{thm2}.
\end{eg}

\begin{eg}
    Let us consider the case
    $$X =
    \begin{xy}
        {(0,-6) \ar @{{*}-o} (4,-2)},
        {(4,-2) \ar @{o.o} (8,2)},
        {(8,2) \ar @{o-o} (12,6)},
        {(12,6) \ar @{o-o} (16,2)},
        {(16,2) \ar @{o.o} (20,-2)},
        {(20,-2) \ar @{o-{*}} (24,-6)},
        {(0,-5) \ar @/^2mm/ @{-}^{n} (7,2)},
        {(17,2) \ar @/^2mm/ @{-}^{m} (24,-5)}
    \end{xy}.$$
    In other words,
    \begin{align*}
        X & \coloneqq \{x_1, \ldots, x_n, x_{n+1}, x_{1}', \ldots, x_{m}'\}\\
        \preceq & \coloneqq \{(x_a, x_b)\,|\,a\leq b\}\cup \{(x_{a}', x_{b}')\,|\,a\leq b\}\cup \{(x_{a}', x_{n+1})\,|\,a \in \{1, \ldots, m\}\}\\
        \delta(x) & \coloneqq \begin{cases}
        1 \text{ if } x \in \{x_{1}, x_1'\}\\
        0 \text{ others.}
    \end{cases}
    \end{align*}
    By Definition \ref{def}, we have
    $$I^\mfk(X) = \sum_{i=1}^{n}\binom{n-i+m}{n-i}\zeta^\mfk(i,n-i+m+1) + \sum_{j=1}^{m}\binom{n+m-j}{m-j}\zeta^\mfk(j,n+m-j+1).$$
    Let us calculate $\Delta(I^\mfk(X))$ using Theorem \ref{thm2}. First, using Proposition \ref{propMI} it is easy to show the formula
    $$I^\mfk(0;\{0\}^a,1,\{0\}^b;1) = (-1)^{a+1}\binom{a+b}{a}\zeta^\mfk(a+b+1),$$
    and using $\binom{i}{k} = \binom{i-1}{k} + \binom{i-1}{k-1}$, it is easy to show the formula
    $$\sum_{k = 0}^{i} (-1)^{i-k+1} \frac{m}{m+k}\binom{i}{k} = \frac{(-1)^{i+1}}{\binom{i+m}{m}}.$$
    Next, using the fact in Remark \ref{rem}, we get
    $$\Delta(I^\mfk(X)) = \sum_{Y\subset X} \Delta_Y(I^\mfk(X)).$$
    Finally, using Theorem \ref{thm2}, Proposition \ref{propYI} and (\ref{eq:01}), we see that $\Delta_Y$ vanishes except for the following two types of $Y$'s.
    \begin{enumerate}[Type 1.]
        \item $Y_{1,i} = \{x_i, \ldots, x_{n+1}, x_{1}', \ldots, x_{m}'\}$.
        \item $Y_{2,j} = \{x_1, \ldots, x_{n+1}, x_{j}', \ldots, x_{m}'\}$.
    \end{enumerate}
    Since $I^\mfk(X_{\widehat{Y_{1,i}}})$ and $I^\mfk(X_{\widehat{Y_{2,j}}})$ are single zeta values, hence we can express the formula of $\Delta(I^\mfk(X))$ in terms of two types of $Y$ and single zeta values as
    $$\Delta(I^\mfk(X)) = 1\otimes I^\mfk(X) + I^\afk(X) \otimes 1 + \sum_{i = 3}^{n+1} I^\afk(Y_{1,i})\otimes (-\zeta^\mfk(i-1)) + \sum_{j = 3}^{m + 1} I^\afk(Y_{2,j}) \otimes (-\zeta^\mfk(j-1)).$$
    By Definition \ref{def} and the formula above we have
    \begin{align*}
        I^\afk(Y_{1,n+1-i}) & = \sum_{k = 0}^{i} \binom{m-1+k}{k} I^\afk(0;\{0\}^{i-k},1,\{0\}^{m+k};1)\\
        & = \sum_{k = 0}^{i} \frac{m}{m+k}\binom{m+k}{k} (-1)^{i-k+1}\binom{m+i}{i-k}\zeta^\afk(m+i+1)\\
        & = \binom{m+i}{m} \left( \sum_{k = 0}^{i} (-1)^{i-k+1} \frac{m}{m+k}\binom{i}{k} \right) \zeta^\afk(m+i+1)\\
        & = \binom{m+i}{m} \frac{(-1)^{i+1}}{\binom{i+m}{m}} \zeta^\afk(m+i+1)\\
        & = (-1)^{i+1}\zeta^\afk(m+i+1).
    \end{align*}
    Hence, we get
    \begin{align*}
        \Delta(I^\mfk(X)) = & 1\otimes I^\mfk(X) + I^\afk(X) \otimes 1\\
        & + \sum_{i = 0}^{n - 2} (-1)^{i}\zeta^\afk(m+i+1) \otimes \zeta^\mfk(n-i)\\
        & + \sum_{j = 0}^{m - 2} (-1)^{j}\zeta^\afk(n+j+1) \otimes \zeta^\mfk(m-j).
    \end{align*}
\end{eg}

\end{document}